%% file: Manuscript_final_P_Kritzer_Dec_2022.tex
\documentclass[11pt,a4paper]{article}
\usepackage{latexsym,amsfonts,amsmath,graphics}
\usepackage{epsfig}
\usepackage{enumitem}
\usepackage[usenames,dvipsnames,svgnames,table]{xcolor}

\usepackage{algorithmic}

\input{common-macros}

\input{common-macros2}

\newcommand{\supp}{\operatorname{supp}}

\newcommand{\tpmod}[1]{{\;(\operatorname{mod}\;#1)}}

\allowdisplaybreaks

\begin{document}

\title{A note on the CBC-DBD construction of lattice rules with general positive weights}

\author{Peter Kritzer}

\date{\today}

\maketitle

\begin{abstract}
    \noindent Lattice rules are among the most prominently studied quasi-Monte Carlo methods to approximate multivariate integrals.
    A rank-$1$ lattice rule to approximate an $s$-dimensional integral is fully specified by its \emph{generating vector} $\bsz \in \bbZ^s$ and its number of points~$N$.
    While there are many results on the existence of ``good'' rank-$1$ lattice rules, there are no explicit constructions of good generating vectors for dimensions $s \ge 3$. 
    This is why one usually resorts to computer search algorithms. In a recent paper by Ebert et al. in the Journal of Complexity, we showed a component-by-component digit-by-digit (CBC-DBD) 
    construction for good generating vectors of rank-1 lattice rules for integration of functions in weighted Korobov classes. However, the result in that paper was 
    limited to product weights. In the present paper, we shall generalize this result to arbitrary positive weights, thereby answering an open question posed 
    in the paper of Ebert et al. We also include a short section on how the algorithm can be implemented in the case of POD weights, by which we see that the CBC-DBD construction is competitive with the classical CBC construction.
\end{abstract}

\noindent\textbf{Keywords:} Numerical integration; lattice points; quasi-Monte Carlo methods; weighted function spaces; 
digit-by-digit construction; component-by-component construction; fast construction. 

\noindent\textbf{2020 MSC:} 65D30, 65D32, 41A55, 41A63.

\section{Introduction}

In high-dimensional numerical integration, one frequently uses \textit{quasi-Monte Carlo (QMC)} rules $Q_{N,s}$ 
to efficiently approximate integrals $I_s$, 
\begin{equation*}
	I_s(f):=
	\int_{[0,1]^s} f(\bsx) \rd\bsx
	\quad\approx\quad
	Q_{N,s}(f):=
	\frac{1}{N}\sum_{k=0}^{N-1} \, f(\bsx_k)
\end{equation*}
for a suitably chosen integrand $f$ (usually, we assume that $f$ is an element of a Hilbert or Banach space, 
see below). I.e., a QMC rule is an equal-weight quadrature rule, and it is---opposed to \textit{Monte Carlo} rules---based on 
deterministically chosen integration nodes $\bsx_0,\bsx_1,\ldots,\bsx_{N-1}\in [0,1)^s$. 
A non-trivial question in 
this field is how the nodes of a QMC rule can be chosen in order to guarantee a low integration error. Depending on 
the properties of the integrand $f$ under consideration, two classes of integration node sets have gained most attention 
in the past decades, namely \textit{digital nets and sequences}, introduced by Sobol', Faure, and Niederreiter (see, e.g., 
\cite{DP10,N92}), and \textit{lattice point sets}, introduced by Korobov and Hlawka (see, e.g., \cite{DKP22,SJ94}). In this paper, 
we focus on instances of the latter, namely so-called \textit{rank-1 lattice point sets} yielding \textit{lattice rules} when used in QMC rules. 
These are sets of integration nodes with $N$ points
\[
 \bsx_k :=
  \left( \left\{ \frac{k z_1}{N}\right\} , \ldots , \left\{ \frac{k z_s}{N}\right\} \right)
  \in [0,1)^s,  \qquad
  \text{for } k \in\{ 0, 1, \ldots, N-1\},
\]
and where $\{x\}=x-\lfloor x \rfloor$ denotes the fractional part of a real $x$.
Note that, given $N$ and $s$, the lattice rule is completely determined by the choice of the \emph{generating vector} $\bsz=(z_1,\ldots,z_s) \in \bbZ_N^s$, where $\bbZ_N := \{0, \ldots, N-1\}$. 
We remark that it is sufficient to consider the choice of $z_j$ modulo~$N$ since $\{k z_j/N\} = (k z_j \bmod{N})/N$ for integer $k$, $N$, and $z_j$. 
Obviously, not every choice of a generating vector $\bsz$ also yields a lattice rule of good quality for approximating the integral. 
For dimensions $s\le 2$, explicit constructions of good generating vectors are available, see, e.g., \cite{DKP22, N92, SJ94}, 
but there are no explicit constructions of good generating vectors known for $s>2$. Therefore, one usually studies search algorithms 
for generating vectors of good lattice rules, which are designed to make a certain error criterion sufficiently small. 
Korobov \cite{Kor63}, and later Sloan and his collaborators \cite{SR02}, introduced a \emph{component-by-component (CBC) construction}, 
which is a greedy algorithm constructing the components $z_1,\ldots,z_s$ of $\bsz \in \bbZ_N^s$ successively, choosing one $z_j$ at a time, 
and keeping previous components fixed. It was shown in \cite{K03} for prime~$N$ and in \cite{D04} for composite $N$ that the CBC construction 
yields generating vectors with essentially optimal convergence rates for particular spaces of integrands. For suitable choices of function spaces, 
there exist fast implementations of CBC constructions, which have a run-time of order $\mathcal{O} (s N \log N)$, see \cite{NC06b,NC06}.

In this short paper, we consider an alternative to the common CBC search algorithm, namely a so-called CBC-DBD algorithm, which is based on 
the idea that the single components of the generating vector $\bsz$ are chosen one after another, and each component is constructed \textit{digit-by-digit (DBD)}, 
i.e., the base-2 digits of the components are chosen in a greedy fashion, starting with the least significant digit. The principle idea of this construction is 
due to Korobov (see \cite{Kor63}), and it was shown to work for a modern function space setting in the recent paper \cite{EKNO21}. However, 
due to technical difficulties, the main result in \cite{EKNO21} does not hold in full generality, but only with certain restrictions on the \textit{weight} parameters 
of the function space involved. In the present paper, we close this gap (see below for a more detailed explanation). 

In order to make use of a construction method like the CBC or the CBC-DBD construction, one needs to define the search criterion which the algorithm 
is based on. Usually, this criterion is related to the class of integrands under consideration. In \cite{EKNO21}, we consider a Banach space (called \textit{Korobov space} 
or \textit{Korobov class}) of functions.

We consider integrands $f$ with absolutely converging Fourier series,
\[
  f(\bsx)=
  \sum_{\bsm \in \Z^s} \hat{f}(\bsm) \, \rme^{2 \pi \icomp \bsm \cdot \bsx}
  \qquad \text{with} \qquad
  \hat{f}(\bsm):=
  \int_{[0,1]^s} f(\bsx) \, \rme^{-2 \pi \icomp \bsm \cdot \bsx} \rd \bsx ,
\]
i.e., $\hat{f}(\bsm)$ is the $\bsm$-th Fourier coefficient of $f$, where $\bsm \cdot \bsx := \sum_{j=1}^s m_j x_j$ is the vector dot product.
Since the Fourier series are absolutely summable, the Fourier series are pointwise convergent, $1$-periodic, and continuous.
Here, we will consider Banach spaces which are based on assuming sufficient decay of the Fourier coefficients of its elements to guarantee certain smoothness. 
These spaces will be denoted by $E_{s,\bsgamma}^{\alpha}$, where $s$ denotes the number of variables the functions depend on, $\alpha>1$ 
is a real number frequently referred to as the smoothness parameter, and $\bsgamma = \{\gamma_\setu\}_{\setu \subset \bbN}$, 
is a set of strictly positive \textit{weights} to model the importance of different subsets of components. In this context, 
by the notation ``$\setu \subset \bbN$'' we mean all finite subsets $\setu$ of the positive integers. We extend the range of $\setu$ 
to all finite subsets of $\bbN$ since we also would like to include results that hold asymptotically when $s$ tends to infinity.

Intuitively, a large $\gamma_{\setu}$ corresponds to a high influence of the variables $x_j$ with $j\in\setu$, 
while a small $\gamma_{\setu}$ means low influence. The idea of weights goes back to Sloan and Wo\'{z}niakowski \cite{SW98},
and will be made more precise by incorporating the weights in the norm of the space $E_{s,\bsgamma}^{\alpha}$ below. 
We are interested in conditions on the weights such that we can bound the integration error independently of $s$.
This is called \emph{strong polynomial tractability}, see, e.g., \cite{NW08}, for a general reference.

In the literature on QMC methods, the weights $\bsgamma$ are not always chosen as fully general. Indeed, a very common variant is 
to work with an infinite sequence $(\gamma_j)_{j\ge 1}$ and to put $\gamma_\setu:=\prod_{j\in\setu} \gamma_j$, which is the case of \textit{product weights}. 
Other variants are, e.g., finite-order weights or product-and-order-dependent (POD) weights, and we refer the reader to \cite{DKS13} for a more detailed discussion. 
The main result in \cite{EKNO21} is shown for the special case of product weights, due to one step in the proof that we were not able to carry out for more general weights. 
In the present paper, we close this gap, and show a corresponding result for general weights $\bsgamma$. The only assumption we make is that 
all $\gamma_\setu$ are positive, which is to avoid too much technical notation. Presumably, a similar result also holds for weights that are allowed to be zero.

Now, for a given smoothness parameter $\alpha > 1$ and strictly positive weights $\{\gamma_\setu\}_{\setu \subset  \bbN}$, we define, 
for any $\bsm=(m_1,\ldots,m_s) \in \Z^s$,
\[
	r_{\alpha,\bsgamma}(\bsm)
	:=
	\gamma_{\supp(\bsm)}^{-1} \prod_{j\in\supp(\bsm)} \abs{m_j}^\alpha ,
\]
where $\supp(\bsm) := \{ j \in \{1,\ldots, s\}: m_j \ne 0 \}$ is the support of $\bsm$. We set $\gamma_\emptyset = 1$, so $r_{\alpha,\bsgamma}(\bszero) = 1$. 
Using this notation, we define the norm of our Banach space $E_{s,\bsgamma}^{\alpha}$,
\begin{equation}\label{eq:norm}
	\|f\|_{E_{s,\bsgamma}^{\alpha}} := 
	\sup_{\bsm \in \Z^s} |\hat{f}(\bsm)| \, r_{\alpha,\bsgamma}(\bsm), \quad \mbox{for}\quad f\in E_{s,\bsgamma}^{\alpha},
\end{equation}
and our weighted function space by
\[
	E_{s,\bsgamma}^{\alpha} :=
	\left\{f \in L^2([0,1]^s) \colon \|f\|_{E_{s,\bsgamma}^{\alpha}}  < \infty \right\}.
\]
The criterion by which we assess the quality of a given rank-1 lattice rule is the \textit{worst-case error}, which is defined as
\[
 e_{N,s,\alpha,\bsgamma}(\bsz):=\sup_{\substack{f\in  E_{s,\bsgamma}^{\alpha}\\ \|f\|_{E_{s,\bsgamma}^{\alpha} \le 1}}}
 \abs{I_s(f) - Q_{N,s}(f,\bsz)},
\]
where we stress the dependence on $\bsz$ in our notation of the error and of the QMC rule $Q_{N,s}$. 
\begin{remark}
Note that since $\alpha>1$, the membership of $f$ to the space $E_{s,\bsgamma}^{\alpha}$ implies the absolute convergence of its Fourier series, 
which in turn entails that $f$ is continuous and $1$-periodic with respect to each variable. In addition, if $f \in E_{s,\bsgamma}^{\alpha}$, $f$
has $1$-periodic continuous mixed partial derivatives $f^{(\bstau)}$ for any $\bstau \in \N_0^s$ with all $\tau_j < \alpha -1$. Furthermore, 
it is known that the optimal convergence rate of the worst-case error in the function space is of order $\mathcal{O}(N^{-\alpha})$. We also remark 
that, by slightly modifying the definition of the norm in \eqref{eq:norm}, one could define a function space that is similar to $E_{s,\bsgamma}^{\alpha}$, 
but a Hilbert space. The latter is frequently studied in the literature on lattice rules (see, e.g., \cite{DKP22}), and the worst-case error in 
$E_{s,\bsgamma}^{\alpha}$ is exactly the square of the worst-case error in the Hilbert case. For consistency with \cite{EKNO21}, we stay with $E_{s,\bsgamma}^{\alpha}$ 
in this note, and refer to that paper for further information on the properties of the function space, as well as for further references. 
\end{remark}

We write $\N := \{1, 2, \ldots\}$ for the set of natural numbers and $\N_0 := \{0,1,2,\ldots\}$, $\bbZ$ for the set of integers and $\Z_N := \{0, \ldots, N-1\}$. 
To denote sets of components we use fraktur font, e.g., $\setu \subset \N$.
As a shorthand we write $\{k_1 : k_2\}$ for the set $\{k_1, k_1+1,\ldots,k_2\}$, 
for two integers $k_1,k_2$ with $k_1\le k_2$. 
To denote the projection of a vector $\bsx \in [0,1)^s$ or $\bsm \in \bbZ^s$ onto the components 
in a set $\setu \subseteq \{1 \mcol s\}$ we write $\bsx_\setu := (x_j)_{j\in\setu}$ or $\bsm_\setu := (m_j)_{j\in\setu}$, respectively.

\section{The CBC-DBD construction yields optimal convergence rates for general weights}

As outlined above, the paper \cite{EKNO21} shows a result for lattice rules with generating vectors obtained by a CBC-DBD algorithm, which implies 
that the lattice rules achieve a convergence order in $E_{s,\bsgamma}^{\alpha}$ that can be arbitrarily close to the optimal rate $N^{-\alpha}$. 
Furthermore, the error bound can be made independent of the dimension if the weights $\bsgamma$ satisfy suitable conditions. While the results 
in \cite{EKNO21} are limited to product weights, we show their generalization to arbitrary positive weights here. To be more precise, several auxiliary 
results in our previous paper hold for general weights, and only some of the proofs there require the assumption of product weights. Here, we will mostly 
highlight those passages where there is a significant difference to what is discussed in \cite{EKNO21}. 
For consistency, we shall use notation that is as similar as possible to that in \cite{EKNO21}, with only few minor adaptions. 

A keystone in the paper \cite{EKNO21} was the following proposition, which implies that in analyzing the worst-case error $e_{N,s,\alpha,\bsgamma}(\bsz)$ it is sufficient 
to consider only a truncated variant of the error expression, namely
\begin{equation}\label{eq:def_Talpha}
	T_{\alpha,\bsgamma}(N,\bsz) 
	:= 
	\sum_{\substack{\bsm \in M_{N,s}\setminus \{\bszero\}\\ \bsm\cdot\bsz\equiv 0 \tpmod{N}}} \frac{1}{r_{\alpha,\bsgamma}(\bsm)}, 
	\quad \mbox{where}\quad M_{N,s}:=\{-(N-1),\ldots,N-1\}^s.
\end{equation}      
Indeed, we have
\begin{proposition}{\cite[Proposition 1]{EKNO21}}\label{prop:trunc_error}
 Let $\bsgamma = \{\gamma_{\setu}\}_{\setu \subset \bbN}$ be positive weights and let 
 $\bsz = (z_1,\ldots,z_s) \in \Z^s$ with $\gcd(z_j,N)=1$ for all $j\in\{1,\ldots,s\}$. Then, for $\alpha>1$, we have that
\[
  e_{N,s,\alpha,\bsgamma}(\bsz) - T_{\alpha,\bsgamma} (N, \bsz) \le
  \frac{1}{N^\alpha} \sum_{\emptyset\neq \setu \subseteq \{1 \mcol s\}} \gamma_\setu \, (4\zeta (\alpha))^{\abs{\setu}},
\]
where $\zeta(\cdot)$ denotes the Riemann zeta function.
\end{proposition}
\begin{proof}
 We refer to \cite{EKNO21} for a proof. 
\end{proof}
Note that one can also define the quantity $T_{\alpha,\bsgamma}(N,\bsz)$ by replacing $\alpha$ by 1 in \eqref{eq:def_Talpha} and making the obvious adaptions. 
Regarding $T_{1,\bsgamma}(N,\bsz)$, the following estimate was shown in \cite{EKNO21} for the case when $N$ is a power of 2.
\begin{theorem}{\cite[Theorem 2]{EKNO21}}\label{thm:T_target_CBCDBD}
Let $N=2^n$, with $n\ge1$, and let $\bsgamma=\{\gamma_\setu\}_{\setu \subset \bbN}$ be positive weights. 
Furthermore, let $\bsz = (z_1, \ldots, z_s)\in\{1,\ldots,N-1\}^s$
with $\gcd (z_j,N)=1$ for $1\le j\le s$. Then,
\begin{eqnarray*}
 T_{1,\bsgamma}(N,\bsz) &\le&
 \sum_{\emptyset \ne \setu \subseteq \{1 \mcol s\}} \frac{\gamma_\setu}{N} (\log 4 + 2(1 + \log N))^{\abs{\setu}} -
 \sum_{\emptyset \ne \setu \subseteq \{1 \mcol s\}} \gamma_\setu \, (\log 4)^{|\setu|} \\
 &&+
 \sum_{\emptyset \ne \setu \subseteq \{1 \mcol s\}} \frac{\gamma_\setu}{N} \, 2 |\setu| \, \left(1 + 2 \log N\right)^{|\setu|} \, (1+\log N) + \frac{1}{N}\, H_{s,n,\bsgamma} (\bsz),
\end{eqnarray*}
where 
\[
  H_{s,n,\bsgamma}(\bsz) :=
  \sum_{k=1}^{2^n-1}
  \sum_{\emptyset \ne \setu \subseteq \{1 \mcol s\}} \gamma_\setu
  \prod_{j\in\setu} \log \left( \frac1{\sin^{2}(\pi k z_j / 2^n)} \right).
\]
\end{theorem}
Based on the quantity $H_{s,n,\bsgamma}(\bsz)$, an averaging argument was used in \cite{EKNO21} to obtain a quality function which serves 
as the relevant criterion in the CBC-DBD algorithm. We state its definition 
here for completeness. 
\begin{definition}[Digit-wise quality function] \label{def:h_rv}
	Let $x \in \N$ be an odd integer, let $n, s \in \N$ be positive integers, and let $\bsgamma=\{\gamma_\setu\}_{\setu \subset \bbN}$
	be positive weights. For $1 \le v \le n$ and $1 \le r \le s$, and odd integers $z_1,\ldots,z_{r-1} \in \Z$, we define the 
	quality function $h_{r,n,v,\bsgamma}$ for odd integer $x$ as
	\begin{align*} 
		h_{r,n,v,\bsgamma}(x) 
		&:=
		\sum_{t=v}^{n} \frac{1}{2^{t-v}} \sum_{\substack{k=1 \\ k \equiv 1 \tpmod{2}}}^{2^t-1} 
		\left[ \sum_{\emptyset \ne \setu\subseteq \{1 \mcol r-1\}} \gamma_{\setu} 
		\prod_{j\in\setu} \log \left( \frac{1}{\sin^2(\pi k z_j / 2^t)} \right) \right.\nonumber\\ 
		&\qquad\qquad\qquad\qquad\qquad+
		\left.\sum_{\setu\subseteq \{1 \mcol r-1\}} \!\!\!\! \gamma_{\setu\cup \{r\}} 
		\log \left( \frac{1}{\sin^2(\pi k x / 2^v)} \right) \prod_{j\in\setu} \log \left( \frac{1}{\sin^2(\pi k z_j / 2^t)}\right) \right]
		.
	\end{align*} 
\end{definition}
Note that while the quantity $h_{r,n,v,\bsgamma}$ depends on the integers $z_1,\ldots,z_{r-1}$, this dependency is not explicitly visible in our notation. 
Nevertheless, in the following these integers will always be the components of the generating vector which have been selected in the previous steps of our algorithm. 
Based on $h_{r,n,v,\bsgamma}$ the component-by-component digit-by-digit (CBC-DBD) algorithm is formulated as follows in \cite{EKNO21}.
\begin{algorithm}[H] 
	\caption{Component-by-component digit-by-digit construction, \cite[Algorithm 1]{EKNO21}}	
	\label{alg:cbc-dbd}
	\vspace{5pt}
	\textbf{Input:} Integer $n \in \N$, dimension $s$, and positive weights $\bsgamma=\{\gamma_\setu\}_{\setu \subset \bbN}$. \\
	\vspace{-7pt}
		\begin{algorithmic}
			\STATE Set $z_{1,n} = 1$ and $z_{2,1} = \cdots = z_{s,1} = 1$.
			\vspace{5pt}
			\FOR{$r=2$ \TO $s$}
			\FOR{$v=2$ \TO $n$}
			\STATE $z^{\ast} = \underset{z \in \{0,1\}}{\argmin} \; h_{r,n,v,\bsgamma}(z_{r,v-1} + z \, 2^{v-1})$
			\STATE $z_{r,v} = z_{r,v-1} + z^{\ast} \, 2^{v-1}$
			\ENDFOR
			\STATE Set $z_r := z_{r,n}$.
			\ENDFOR
			\vspace{5pt}
			\STATE Set $\bsz = (z_1,\ldots,z_s)$.
		\end{algorithmic}
	\vspace{5pt}
	\textbf{Return:} Generating vector $\bsz = (z_1,\ldots,z_s)$ for $N=2^n$.
\end{algorithm}

In the analysis of the worst-case error of rank-1 lattice rules generated by vectors $\bsz$ obtained 
from Algorithm \ref{alg:cbc-dbd},the following theorem was shown in \cite{EKNO21}.
\begin{theorem}{\cite[Theorem 3]{EKNO21}} \label{theorem:H-induction}
Let $n, s \in \N$, $N=2^n$, and let $\bsgamma=\{\gamma_\setu\}_{\setu \subset \bbN}$ be positive weights with $\gamma_\emptyset=1$. 
Furthermore, let the generating vector $\bsz \in \Z_N^s$ be constructed by Algorithm \ref{alg:cbc-dbd}. Then, 
\begin{equation} \label{eq:estimate-H-induction}
 H_{s,n,\bsgamma}(\bsz) \le 
 H_{s-1,n,\bsgamma}(\bsz_{\{1 \mcol s-1\}}) + (\log 4) \left[ \gamma_{\{s\}} N + H_{s-1,n,\bsgamma \cup \{s\}}(\bsz_{\{1 \mcol s-1\}}) \right],
\end{equation}
where $\bsgamma \cup \setv$ is the collection of weights $\{\gamma_{\setu \cup \setv}\}_{\setu \subset \bbN}$ for a finite $\setv\subset\bbN$. 
\end{theorem} 
\begin{remark}
 Note that, with the notation introduced in Theorem \ref{theorem:H-induction}, we have $(\bsgamma \cup \setv)_{\setu}=\gamma_{\setu \cup \setv}$ 
 for finite $\setu,\setv \subset \bbN$.
\end{remark}

Moreover, the following estimate was derived for the special case of product weights in \cite{EKNO21}. This is the point from which onwards the analysis in \cite{EKNO21} 
is only done for the case of product weights.
\begin{theorem}{\cite[Theorem 4]{EKNO21}} \label{theorem:upper_bound_H}
Let $n, s \in \N$, $N=2^n$, and let $\bsgamma=\{\gamma_\setu\}_{\setu \subset \bbN}$, with $\gamma_{\setu} = \prod_{j \in \setu} \gamma_j$ 
be positive product weights. Furthermore, let the generating vector $\bsz \in \Z_N^s$ be constructed by Algorithm \ref{alg:cbc-dbd}. 
Then for $H_{s,n,\bsgamma}(\bsz)$ the following upper bound holds,
$$
 H_{s,n,\bsgamma}(\bsz) \notag \le
 N \left[ -1 + \prod_{j=1}^s (1 + \gamma_j \log 4) \right].
$$
\end{theorem}

We will now show a generalization of Theorem \ref{theorem:upper_bound_H} to general positive weights $\bsgamma=\{\gamma_\setu\}_{\setu \subset \bbN}$ 
without having to make any further assumptions on their structure. Indeed, we have the following theorem. 
\begin{theorem}\label{theorem:upper_bound_H_gen}
Let $n, s \in \N$, $N=2^n$, and let $\bsgamma=\{\gamma_\setu\}_{\setu \subset \bbN}$ be positive weights. 
Furthermore, let the generating vector $\bsz \in \Z_N^s$ be constructed by Algorithm \ref{alg:cbc-dbd}.
Then for $H_{s,n,\bsgamma}(\bsz)$ the following upper bound holds:
$$
 H_{s,n,\bsgamma}(\bsz) \notag \le
N \sum_{\emptyset \neq \setv \subseteq \{1\mcol s\}}
(\log 4)^{\abs{\setv}} \gamma_{\setv}.
$$
\end{theorem}
\begin{proof}
Due to the formulation of Algorithm \ref{alg:cbc-dbd}, it is such that it is extensible in the dimension, 
and for every $r\in\{2,\ldots,s\}$ we have, similarly to Theorem \ref{theorem:H-induction},  
\begin{equation}\label{eq:recursion_H}
 H_{r,n,\bsgamma} (\bsz_{\{1\mcol r\}}) \le H_{r-1,n,\bsgamma} (\bsz_{\{1\mcol r-1\}}) + 
 (\log 4) [\gamma_{\{r\}} N + H_{r-1,n,\bsgamma\cup\{r\}} (\bsz_{\{1\mcol r-1\}})].
\end{equation}
Thus, we obtain
\begin{eqnarray*}
 H_{s,n,\bsgamma} (\bsz_{\{1\mcol s\}}) &\le& H_{s-1,n,\bsgamma} (\bsz_{\{1\mcol s-1\}}) + 
 (\log 4) [\gamma_{\{s\}} N + H_{s-1,n,\bsgamma\cup\{s\}} (\bsz_{\{1\mcol s-1\}})]\\
 &\le& H_{s-2,n,\bsgamma}(\bsz_{\{1 \mcol s-2\}}) + 
  (\log 4) [\gamma_{\{s-1\}} N + H_{s-2,n,\bsgamma\cup\{s-1\}} (\bsz_{\{1 \mcol s-2\}})]\\
&&+(\log 4) [\gamma_{\{s\}} N + H_{s-1,n,\bsgamma\cup\{s\}} (\bsz_{\{1 \mcol s-1\}})]\\
&\le& H_{s-2,n,\bsgamma}(\bsz_{\{1 \mcol s-2\}}) +  (\log 4) \gamma_{\{s-1\}} N + (\log 4)\gamma_{\{s\}} N\\
&&+ (\log 4) H_{s-2,n,\bsgamma\cup\{s-1\}} (\bsz_{\{1 \mcol s-2\}})\\
&&+ (\log 4) \left[H_{s-2,n,\bsgamma\cup\{s\}} (\bsz_{\{1 \mcol s-2\}}) + (\log 4) \gamma_{\{s-1, s\}} N \right.\\
&& + \left.(\log 4) H_{s-2,n,\bsgamma\cup\{s-1, s\}}(\bsz_{\{1 \mcol s-2\}})\right]\\
&=& H_{s-2,n,\bsgamma}(\bsz_{\{1 \mcol s-2\}}) \\
&&+ (\log 4) H_{s-2,n,\bsgamma\cup\{s-1\}} (\bsz_{\{1 \mcol s-2\}}) + (\log 4) H_{s-2,n,\bsgamma\cup\{s\}} (\bsz_{\{1 \mcol s-2\}})\\
&&+ (\log 4)^2 H_{s-2,n,\bsgamma\cup\{s-1, s\}}(\bsz_{\{1 \mcol s-2\}})\\
&&+ (\log 4) \gamma_{\{s-1\}} N + (\log 4) \gamma_{\{s\}} N + (\log 4)^2 \gamma_{\{s-1,s\}} N.
\end{eqnarray*}
Consequently, we have shown
\[
 H_{s,n,\bsgamma} (\bsz_{\{1\mcol s\}}) \le \sum_{\setv \subseteq \{s-1 \mcol s\}} (\log 4)^{\abs{\setv}} H_{s-2,n,\bsgamma\cup\setv} (\bsz_{\{1 \mcol s-2\}}) + 
 N \sum_{\emptyset\neq \setv\subseteq \{s-1 \mcol s\}} (\log 4)^{\abs{\setv}} \gamma_{\setv}.
\]
This yields, by using \eqref{eq:recursion_H} again,
\begin{eqnarray*}
 H_{s,n,\bsgamma} (\bsz_{\{1 \mcol s\}}) &\le& \sum_{\setv \subseteq \{s-1 \mcol s\}} (\log 4)^{\abs{\setv}} 
 \left[ H_{s-3,n,\bsgamma\cup\setv} (\bsz_{\{1 \mcol s-3\}}) + (\log 4) (\bsgamma\cup\setv)_{\{s-2\}} N \right.\\
 &&+\left. (\log 4) H_{s-3,n,\bsgamma\cup\setv\cup\{s-2\}} (\bsz_{\{1 \mcol s-3\}})\right] 
 + N\sum_{\emptyset\neq \setv\subseteq \{s-1 \mcol s\}} (\log 4)^{\abs{\setv}} \gamma_{\setv}\\
 &=& \sum_{\setv \subseteq \{s-2:s\}} (\log 4)^{\abs{\setv}} H_{s-3,n,\bsgamma\cup\setv} (\bsz_{\{1 \mcol s-3\}}) + 
 N \sum_{\emptyset\neq \setv\subseteq \{s-2:s\}} (\log 4)^{\abs{\setv}} \gamma_{\setv}.
\end{eqnarray*}
We can use this principle recursively, until we arrive at
\[
  H_{s,n,\bsgamma} (\bsz_{\{1 \mcol s\}}) \le \sum_{\setv \subseteq \{2:s\}} (\log 4)^{\abs{\setv}} H_{1,n,\bsgamma\cup\setv} (\bsz_{\{1\}}) + 
 N \sum_{\emptyset\neq \setv\subseteq \{2:s\}} (\log 4)^{\abs{\setv}} \gamma_{\setv}.
\]
Furthermore, it can be shown in complete analogy to \cite[Proof of Theorem 4]{EKNO21} that
\[
  H_{1,n,\bsgamma\cup\setv} (\bsz_{\{1\}}) = (\log 4) \gamma_{\{1\}\cup \setv} (N-n-1)  \le (\log 4) \gamma_{\{1\}\cup \setv} N,
\]
which yields 
\begin{eqnarray*}
   H_{s,n,\bsgamma} (\bsz_{\{1 \mcol s\}}) &\le& N\sum_{\setv \subseteq \{2:s\}} (\log 4)^{\abs{\setv}+1} \gamma_{\{1\}\cup \setv} + 
 N \sum_{\emptyset\neq \setv\subseteq \{2:s\}} (\log 4)^{\abs{\setv}} \gamma_{\setv}\\
 &=& N \sum_{\emptyset\neq\setv\subseteq \{1 \mcol s\}} (\log 4)^{\abs{\setv}} \gamma_{\setv},
\end{eqnarray*}
as claimed. 
\end{proof}

We can now use the general estimate in Theorem \ref{theorem:upper_bound_H_gen} to also show a more general version of 
\cite[Theorem 5]{EKNO21}. 

\begin{theorem}\label{thm:main_result_gen}
Let $N=2^n$, with $n \in \N$, and let $\bsgamma=\{\gamma_\setu\}_{\setu\subset \NN}$ be general positive weights. 
Furthermore, denote by $\bsz = (z_1, \ldots, z_s)$ the corresponding generating vector constructed by Algorithm \ref{alg:cbc-dbd}. Then there exists a 
constant $a>0$, which is independent of $s$ and $N$, such that
\begin{equation}\label{eq:optcoeff-dbd_gen}
	T_{1,\bsgamma}(N,\bsz) \le	(1+\log N)\, \sum_{\emptyset \ne \setu \subseteq \{1 \mcol s\}} \frac{\gamma_\setu}{N} (a \log N)^{\abs{\setu}}.
\end{equation}
Moreover, if the weights satisfy	 	
\[
	\sum_{j \ge 1} \tilde\gamma_j < \infty,
		\quad\text{where}\quad
		\tilde\gamma_j := \max_{\setv\subseteq\{1 \mcol j-1\}} \frac{\gamma_{\setv\cup\{j\}}}{\gamma_\setv},
\]
then for any $\delta>0$ there exists a constant $C_\delta >0$, which is again independent of $s$ and $N$, such that 
\[
 T_{1,\bsgamma}(N,\bsz) \le \frac{C_\delta}{N^{1-\delta}}.
\]
\end{theorem}
\begin{proof}
We use the bound in Theorem \ref{thm:T_target_CBCDBD} combined with the bound on $H_{s,n,\bsgamma}(\bsz)$ in Theorem \ref{theorem:upper_bound_H_gen} to obtain
\begin{eqnarray*}
	T_{1,\bsgamma}(N,\bsz) &\le &
	 \sum_{\emptyset \ne \setu \subseteq \{1 \mcol s\}} \frac{\gamma_\setu}{N} (\log 4 + 2(1 + \log N))^{\abs{\setu}} \\
	&&+\sum_{\emptyset \ne \setu \subseteq \{1 \mcol s\}} \frac{\gamma_\setu}{N} \, 2 |\setu| \, \left(1 + 2 \log N\right)^{|\setu|} \, (1+\log N)\\
	&\le &
	 \sum_{\emptyset \ne \setu \subseteq \{1 \mcol s\}} \frac{\gamma_\setu}{N} (\log 4 + 2(1 + \log N))^{\abs{\setu}} \\
	&&+\sum_{\emptyset \ne \setu \subseteq \{1 \mcol s\}} \frac{\gamma_\setu}{N} \, 2^{|\setu|+1} \, \left(1 + 2 \log N\right)^{|\setu|} \, (1+\log N),
\end{eqnarray*} 
where we used that $\abs{\setu}\le 2^{\abs{\setu}}$ for non-empty $\setu$.
Thus, there exists a constant $a>0$ such that
\[
 T_{1,\bsgamma}(N,\bsz) \le (1+\log N)\, \sum_{\emptyset \ne \setu \subseteq \{1 \mcol s\}} \frac{\gamma_\setu}{N} (a \log N)^{\abs{\setu}},
\]
which is the bound in \eqref{eq:optcoeff-dbd_gen}. The second claim in the theorem follows immediately by using \cite[Lemma 4]{EKNO21}. 
\end{proof}

Next, we use Theorem \ref{thm:main_result_gen} to show a more general version of \cite[Corollary 1]{EKNO21}. We formulate 
this result for weights denoted by $\bseta$, and then describe two special cases in a remark below.
\begin{corollary} \label{cor:main-result-dbd_gen}
Let $N=2^n$, with $n \in \N$, and let $\bseta=\{\eta_\setu\}_{\setu\subset \bbN}$ be general positive weights, satisfying
\begin{equation} \label{eq:summability-gamma-tilde}
	\sum_{j \ge 1} \tilde\eta_j < \infty,
		\quad\text{where}\quad
		\tilde\eta_j := \max_{\setv\subseteq\{1 \mcol j-1\}} \frac{\eta_{\setv\cup\{j\}}}{\eta_\setv}.
\end{equation}
Denote by $\bsz=(z_1,\ldots,z_s)$ the generating vector constructed by Algorithm \ref{alg:cbc-dbd} for the weights $\bseta$.
Then, for any $\delta>0$ and each $\alpha>1$, the worst-case error $e_{N,s,\alpha,\bseta^{\alpha}}(\bsz)$ 
in $E_{s,\bseta^\alpha}^{\alpha}$
satisfies
\begin{equation*}
	e_{N,s,\alpha,\bseta^{\alpha}}(\bsz)
	\le \frac{C}{N^\alpha} + \frac{C_\delta^{\alpha}}{N^{\alpha(1-\delta)}}
\end{equation*}  
with constants $C_\alpha,C_\delta > 0$, where $C_\alpha$ may depend on $\alpha$, and $C_\delta$ may depend on $\delta$, 
but $C_\alpha$ and $C_\delta$ are independent of $s$ and $N$. 
\end{corollary}
\begin{proof}
Since $N=2^n$ and since by the formulation of \RefAlg{alg:cbc-dbd} all components of $\bsz$ are odd, we have in particular that $\gcd(z_j,N)=1$ for all $j\in\{1,\ldots,s\}$.
Therefore, by Proposition \ref{prop:trunc_error}, the worst-case error $e_{N,s,\alpha,\bseta^\alpha}(\bsz)$ satisfies
\[
  e_{N,s,\alpha,\bseta^\alpha}(\bsz) \le
  \frac{1}{N^\alpha} \sum_{\emptyset\neq \setu \subseteq \{1 \mcol s\}} \eta_\setu^\alpha \, (4\zeta (\alpha))^{\abs{\setu}}
  + T_{\alpha,\bseta^\alpha}(N,\bsz).
\]
Next, we use that for $\alpha \ge 1$ we have $\sum_i x_i^\alpha \le (\sum_i x_i)^\alpha$ for $x_i \ge 0$, and we also use Theorem \ref{thm:main_result_gen}. 
This yields that for any $\delta>0$ there exists a constant $C_\delta>0$, which is independent of $s$ and $N$, such that
\[
  T_{\alpha,\bseta^\alpha}(N,\bsz)=\!\!\!\sum_{\substack{\bszero \ne \bsm \in M_{N,s}\\ \bsm\cdot\bsz\equiv 0 \tpmod{N}}} \frac{1}{r_{\alpha,\bseta^{\alpha}}(\bsm)}  \le
  \left( \sum_{\substack{\bszero \ne \bsm \in M_{N,s}\\ \bsm\cdot\bsz\equiv 0 \tpmod{N}}} \frac{1}{r_{1,\bseta}(\bsm)} \right)^{\!\!\alpha} =
  (T_{1,\bseta}(N,\bsz))^\alpha \le \frac{C_\delta^{\alpha}}{N^{\alpha(1-\delta)}}.
\]
Using Assumption \eqref{eq:summability-gamma-tilde} and standard arguments (or a similar reasoning as in the proof of \cite[Lemma 4]{EKNO21}), we also see that 
\[
 \sum_{\emptyset\neq \setu \subseteq \{1 \mcol s\}} \eta_\setu \, (4\zeta (\alpha))^{\abs{\setu}} < C_\alpha
\]
for some constant $C_\alpha>0$. This yields the result.  
\end{proof}
\begin{remark}
 If we choose $\bseta=\bsgamma=\{\gamma_\setu\}_{\setu\subset \NN}$ for positive weights $\bsgamma$ in Corollary \ref{cor:main-result-dbd_gen}, then 
 the algorithm can be run independently of $\alpha$, and we get an error bound of almost optimal convergence order for the space 
 $E_{s,\bsgamma^\alpha}^{\alpha}$ for any $\alpha>1$. In this sense, Algorithm \ref{alg:cbc-dbd} can be said to be universal with respect to $\alpha$.
 
 If, on the other hand, we would like to have an error bound of almost optimal convergence order for the space $E_{s,\bsgamma}^{\alpha}$, then we need to 
 choose $\bseta=\bsgamma^{1/\alpha}=\{\gamma_\setu^{1/\alpha}\}_{\setu\subset \NN}$ in Corollary \ref{cor:main-result-dbd_gen}, and also run Algorithm \ref{alg:cbc-dbd} 
 for $\bsgamma^{1/\alpha}$, i.e., in this case Algorithm \ref{alg:cbc-dbd} is not universal.
\end{remark}

\section{Efficient implementation of Algorithm \ref{alg:cbc-dbd} for POD weights}

In \cite[Section 4]{EKNO21}, it was outlined how Algorithm \ref{alg:cbc-dbd} can be efficiently implemented for the case of product weights, and 
it was shown that the computation time of such an implementation is of order $\calO (s N \log N)$. This makes the implementation competitive with 
the fast implementation of the classical CBC algorithm, as proposed by Nuyens and Cools, see, e.g., \cite{NC06, NC06b}, and also \cite{DKP22}.

In this section, we will outline that the CBC-DBD algorithm is also competitive for the case of 
so-called POD (product-and-order-dependent) weights, which occur in applications of QMC methods to PDEs with random coefficients, see, e.g., \cite{KN16}. 
Let us therefore assume that the weights $\bsgamma=\{\gamma_\setu\}_{\setu\subset \NN}$ are POD weights, i.e., they are of the form
\begin{equation}\label{eq:def_POD}
 \gamma_{\setu}=\Gamma_{\abs{\setu}} \prod_{j\in\setu} \gamma_j,\quad \mbox{for}\ \setu\subset\NN,
\end{equation}
where $(\gamma_j)_{j\ge 1}$ and $\Gamma_{\setu}$, $\setu\subset\NN$, are positive reals. We set $\Gamma_{0}=1$ and also treat the empty product as 
one, such that $\gamma_{\emptyset}=1$. It is known that there is a fast implementation of the classical CBC algorithm for POD weights, needing $\calO (s^2 N\log N)$ operations, see again \cite{DKP22}. We will now show that a similar estimate holds for a suitable implementation of Algorithm \ref{alg:cbc-dbd}. 
In the case of POD weights, we obtain from Definition \ref{def:h_rv},
\begin{eqnarray*}
 \lefteqn{h_{r,n,v,\bsgamma}(x) 
        :=
		\sum_{t=v}^{n} \frac{1}{2^{t-v}} \sum_{\substack{k=1 \\ k \equiv 1 \tpmod{2}}}^{2^t-1} 
		\left[ \sum_{\emptyset \ne \setu\subseteq \{1 \mcol r-1\}} \gamma_{\setu} 
		\prod_{j\in\setu} \log \left( \frac{1}{\sin^2(\pi k z_j / 2^t)} \right) \right.}\\ 
		&&+
		\sum_{\emptyset\neq\setu\subseteq \{1 \mcol r-1\}} \!\!\!\! \gamma_{\setu\cup \{r\}} 
		\log \left( \frac{1}{\sin^2(\pi k x / 2^v)} \right) \prod_{j\in\setu} \log \left( \frac{1}{\sin^2(\pi k z_j / 2^t)}\right) \\
		&&+ \left.\Gamma_1 \gamma_{r} \log \left( \frac{1}{\sin^2(\pi k x / 2^v)}\right)\vphantom{\sum_{\emptyset \ne \setu\subseteq \{1 \mcol r-1\}}}
		\right]\\
		&=& \sum_{t=v}^{n} \frac{1}{2^{t-v}} \sum_{\substack{k=1 \\ k \equiv 1 \tpmod{2}}}^{2^t-1} 
		\left[\vphantom{\left(\sum_{\substack{\setu\subseteq \{1:r-1\}\\ \abs{\setu}=\ell}}\right)}
		\Gamma_1 \gamma_{r} \log \left( \frac{1}{\sin^2(\pi k x / 2^v)}\right)\right.\\
		&&+\sum_{\ell=1}^{r-1}\left(\sum_{\substack{\setu\subseteq \{1:r-1\}\\ \abs{\setu}=\ell}} \Gamma_{\ell} 
		\prod_{j\in\setu}\gamma_j \log \left( \frac{1}{\sin^2(\pi k z_j / 2^t)}\right)\right.\\
		&&+\left.\left.\frac{\Gamma_{\ell+1}}{\Gamma_{\ell}} 
		\gamma_{r} \log \left( \frac{1}{\sin^2(\pi k x / 2^v)}\right)\sum_{\substack{\setu\subseteq \{1:r-1\}\\ \abs{\setu}=\ell}} \Gamma_{\ell} 
		\prod_{j\in\setu}\gamma_j \log \left( \frac{1}{\sin^2(\pi k z_j / 2^t)}\right)\right)\right]\\
		&=& \sum_{t=v}^{n} \frac{1}{2^{t-v}} \sum_{\substack{k=1 \\ k \equiv 1 \tpmod{2}}}^{2^t-1} 
		\left[\vphantom{\left(\sum_{\substack{\setu\subseteq \{1:r-1\}\\ \abs{\setu}=\ell}}\right)}
		\Gamma_1 \gamma_{r} \log \left( \frac{1}{\sin^2(\pi k x / 2^v)}\right)\right.\\
		&&+ \sum_{\ell=1}^{r-1}\sum_{\substack{\setu\subseteq \{1:r-1\}\\ \abs{\setu}=\ell}}
		\Gamma_{\ell} \prod_{j\in\setu}\gamma_j \log \left( \frac{1}{\sin^2(\pi k z_j / 2^t)}\right)\\
		&&\left.+ \gamma_r \log \left( \frac{1}{\sin^2(\pi k x / 2^v)}\right) 
		\sum_{\ell=1}^{r-1}\frac{\Gamma_{\ell+1}}{\Gamma_{\ell}}\sum_{\substack{\setu\subseteq \{1:r-1\}\\ \abs{\setu}=\ell}}
		\Gamma_{\ell} \prod_{j\in\setu}\gamma_j \log \left( \frac{1}{\sin^2(\pi k z_j / 2^t)}\right)\right].
\end{eqnarray*}
For short, we write, for $r\ge 2$, $\ell\in \{1,\ldots,r-1\}$, $t\in\{v,\ldots,n\}$, and odd $k\in \{1,\ldots, 2^t-1\}$,
\[
 p_{r-1,\ell,t}(k):=\sum_{\substack{\setu\subseteq \{1:r-1\}\\ \abs{\setu}=\ell}}
		\Gamma_{\ell} \prod_{j\in\setu}\gamma_j \log \left( \frac{1}{\sin^2(\pi k z_j / 2^t)}\right),
\]
and note that $p_{r-1,\ell,t}(k)$ is independent of $x$. Thus we can write
\begin{multline}\label{eq:quality_with_p}
 h_{r,n,v,\bsgamma}(x) = \sum_{t=v}^{n} \frac{1}{2^{t-v}} \sum_{\substack{k=1 \\ k \equiv 1 \tpmod{2}}}^{2^t-1} 
		\left[
		\Gamma_1 \gamma_{r} \log \left( \frac{1}{\sin^2(\pi k x / 2^v)}\right)\right.\\
		\left. + \sum_{\ell=1}^{r-1} p_{r-1,\ell,t} (k) + \gamma_r \log \left( \frac{1}{\sin^2(\pi k x / 2^v)}\right) 
		\sum_{\ell=1}^{r-1} \frac{\Gamma_{\ell+1}}{\Gamma_\ell} p_{r-1,\ell,t}(k)\right].
\end{multline}
Note that we have, for any $r\ge 2$, 
\begin{equation}\label{eq:p_recursion}
 p_{r,\ell,t}(k)=p_{r-1,\ell,t}(k) + \frac{\Gamma_{\ell}}{\Gamma_{\ell-1}} \gamma_r 
 \left(\log \left( \frac{1}{\sin^2(\pi k z_r / 2^t)}\right) \right) p_{r-1,\ell-1,t}(k),
\end{equation}
where we set
\begin{equation}\label{eq:initial_values}
 p_{r-1,\ell,t}(k)=\begin{cases}
                    1 & \mbox{if $\ell=0$,}\\
                    0 & \mbox{if $r-1<\ell$,}
                   \end{cases}
\end{equation}
for $t\in\{v,\ldots,n\}$ and odd $k\in\{1,2,\ldots,2^t-1\}$. 

Let
\[
P_{r,\ell,t}(k2^{n-t})=p_{r,\ell,t}(k) =\sum_{\substack{\setu\subseteq \{1:r\}\\ \abs{\setu}=\ell}}
		\Gamma_{\ell} \prod_{j\in\setu}\gamma_j \log \left( \frac{1}{\sin^2(\pi k z_j / 2^t)}\right)
\]
for $t\in\{2,\ldots,n\}$ and corresponding odd indices $k\in \{1,3,\ldots,2^t-1\}$. Note, furthermore, that for the 
evaluation of $h_{r,n,v,\bsgamma}$ we do not require the values of $P_{r,\ell,t}(k2^{n-t})=p_{r,\ell,t}(k)$ for $t\in\{2,\ldots,v-1\}$. 

Additionally, due to the way the $z_{r,v}$ are constructed in Algorithm \ref{alg:cbc-dbd}, we have that 
$z_{r,n} \tpmod{2^v}=z_{r,v}$ for $1\le v\le n$, and thus, by the periodicity of $\sin^2 (\pi x)$,
\[
 \sin^2\left(\pi \frac{k z_r}{2^v}\right)=\sin^2 \left(\pi \frac{k z_{r,n}\tpmod{2^v}}{2^v}\right)
 =\sin^2 \left(\pi \frac{k z_{r,v}}{2^v}\right).
\]
Hence, we can perform the update as in \eqref{eq:p_recursion} for $k\in\{1,3,\ldots,2^v-1\}$ with $z_{r}$ replaced by
$z_{r,v}$ immediately after each $z_{r,v}$ has been determined. 

These observations now lead to the following fast implementation of Algorithm \ref{alg:cbc-dbd}. 
\begin{algorithm}[H] 
	\caption{Fast CBC-DBD algorithm for POD weights}	
	\label{alg:fast-cbc-dbd}
	\vspace{5pt}
	\textbf{Input:} Integer $n \in \N$, dimension $s$, and positive POD weights $\bsgamma=\{\gamma_\setu\}_{\setu \subset \bbN}$. \\
	\vspace{-7pt}
		\begin{algorithmic}
			\FOR{$t=2$ \TO $n$}
			\FOR{$k=1$ \TO $2^t-1$ in steps of 2}
			\STATE Put $P_{1,1,t} (k 2^{n-t})=\Gamma_1 \gamma_1 \log \left(\frac{1}{\sin^2 (\pi k /2^t)}\right)$
			\ENDFOR
			\ENDFOR
			\vspace{5pt}
			\STATE Set $z_1=1$ and $z_{2,1}=\cdots = z_{s,1}=1$
			\FOR{$r=2$ \TO $s$}
			\FOR{$v=2$ \TO $n$}
            \STATE $z^{\ast} = \underset{z \in \{0,1\}}{\argmin} \; h_{r,n,v,\bsgamma} (z_{r,v-1}+ z2^{v-1})$, 
            where $h_{r,n,v,\bsgamma}$ is evaluated using \eqref{eq:quality_with_p}
            \STATE Set $z_{r,v}:=z_{r,v-1} + z^* 2^{v-1}$
            \FOR{$\ell=1$ \TO $r$}
            \FOR{$k=1$ \TO $2^{v}-1$ in steps of 2}
            \STATE $P_{r,\ell,v}(k 2^{n-v})=P_{r-1,\ell,v} (k 2^{n-v}) + \frac{\Gamma_{\ell}}{\Gamma_{\ell -1}} 
                    \log\left(\frac{1}{\sin^2 (\pi k z_{r,v} / 2^v)}\right)  P_{r-1,\ell-1,v}(k)$, where the initial values are 
                    defined analogously to \eqref{eq:initial_values}
            \ENDFOR
            \ENDFOR
			\ENDFOR
			\STATE Set $z_r:=z_{r,n}$
			\ENDFOR
		\end{algorithmic}
	\vspace{5pt}
	\textbf{Return:} Generating vector $\bsz = (z_1,\ldots,z_s)$ for $N=2^n$.
\end{algorithm}

Next, we show the following proposition, which implies that the computation time and required memory
of Algorithm \ref{alg:fast-cbc-dbd} are competitive with the classical fast CBC algorithm for POD weights. 
\begin{proposition}\label{prop:comp_time_POD}
 Let $n,s\in\NN$ and let $N=2^n$. For a given set of positive POD weights $\bsgamma=\{\gamma_\setu\}_{\setu \subset \bbN}$, 
 a generating vector $\bsz = (z_1,\ldots,z_s)$ can be computed via Algorithm \ref{alg:fast-cbc-dbd} using 
 $\calO (s^2 N \log N)$ operations and requiring $\calO (Ns)$ memory. 
\end{proposition}
\begin{proof}
Due to \eqref{eq:quality_with_p}, the cost of evaluating $h_{r,n,v,\bsgamma} (x)$ is of order
\[
  \sum_{t=v}^n 2^{t-1} (r-1).
\]
Updating the $P_{r,\ell,v}$ for $\ell=1,\ldots,r$ and odd $k=1,3,\ldots,2^{v-1}$ costs $\calO (r 2^{v-1})$ operations. 
Hence, the inner loop of Algorithm \ref{alg:fast-cbc-dbd} over $v=2,\ldots,n$ takes 
\begin{eqnarray*}
 \calO \left(\sum_{v=2}^n \left(\sum_{t=v}^n 2^{t-1} (r-1) + r 2^{v-1}\right)\right) &=&
 \calO \left(2^n (2+ n (r-1)-r)\right)\\
 &=& \calO (r N \log N)
\end{eqnarray*}
operations. Consequently, the outer loop over $r=2,\ldots,s$ takes $\calO (s^2 N \log N)$ operations. 

Regarding storage, initializing $P_{1,1,t}$ costs $\calO (N)$ memory. Then, updating $P_{r,\ell,v} (k 2^{n-v})$ costs at most $\calO (N s)$ memory.

\end{proof}

\subsection*{Acknowledgments}

The author is supported by the Austrian Science Fund, Project F5506, 
which is part of the Special Research Program ``Quasi-Monte Carlo Methods: Theory and Applications'', and Project P34808. 
Moreover, the author would like to thank an anonymous referee for their suggestions on how to improve the presentation of the results.
For the purpose of open access, the author has applied a CC BY public copyright licence to any author accepted manuscript version arising from this submission.

\begin{small}
	\noindent\textbf{Author's address:}\\

	\noindent Peter Kritzer\\
	Johann Radon Institute for Computational and Applied Mathematics (RICAM)\\
	Austrian Academy of Sciences\\
	Altenbergerstr. 69, 4040 Linz, Austria.\\
	\texttt{peter.kritzer@oeaw.ac.at}

\end{small}

\end{document}

%% file: common-macros.tex
\usepackage{graphicx}       % standard LaTeX graphics tool
\usepackage[latin1]{inputenc}
\usepackage[T1]{fontenc}
\usepackage{microtype} % good font tricks

\usepackage{cite}
\usepackage{array,colortbl}
\usepackage{amsmath,amsfonts,amssymb,bm,amsthm,mathtools,mathrsfs}
\DeclareFontFamily{U}{mathx}{\hyphenchar\font45}
\DeclareFontShape{U}{mathx}{m}{n}{
      <5> <6> <7> <8> <9> <10>
      <10.95> <12> <14.4> <17.28> <20.74> <24.88>
      mathx10
      }{}
\DeclareSymbolFont{mathx}{U}{mathx}{m}{n}
\DeclareFontSubstitution{U}{mathx}{m}{n}
\DeclareMathAccent{\widecheck}{0}{mathx}{"71}

\def\citep#1#2{\cite[{#1}]{#2}}

\newcommand{\RefAlg}[1]{Algorithm~\textup{\ref{#1}}}

\usepackage{algorithm}
\usepackage{algorithmic}
\newcommand{\bbN}{{\mathbb{N}}}

\newcommand{\bbZ}{{\mathbb{Z}}}
\newcommand{\N}{{\mathbb{N}}} % natural numbers {1, 2, ...}
\newcommand{\Z}{{\mathbb{Z}}} % integers
\newcommand{\NN}{{\mathbb{N}}} % natural numbers {1, 2, ...}
\DeclareSymbolFont{bbold}{U}{bbold}{m}{n}
\DeclareSymbolFontAlphabet{\mathbbold}{bbold}

\newcommand{\bsm}{{\boldsymbol{m}}}

\newcommand{\bsx}{{\boldsymbol{x}}}

\newcommand{\bsz}{{\boldsymbol{z}}}

\newcommand{\bszero}{{\boldsymbol{0}}} % vector of zeros
\newcommand{\bsgamma}{{\boldsymbol{\gamma}}}

\newcommand{\bseta}{{\boldsymbol{\eta}}}

\newcommand{\bstau}{{\boldsymbol{\tau}}}

\newcommand{\calO}{{\mathcal{O}}}

\newcommand{\setu}{{\mathfrak{u}}}
\newcommand{\setv}{{\mathfrak{v}}}

\newcommand{\rme}{{\mathrm{e}}}

\newcommand{\icomp}{\mathrm{i}}
\newcommand{\abs}[1]{\left\vert#1\right\vert}

\newcommand{\rd}{\,\mathrm{d}} % differential symbol with tiny space in front for use in integrals
\providecommand{\argmin}{\operatorname*{argmin}}

\newcommand{\mcol}{{\mathpunct{:}}}

%% file: common-macros2.tex
\usepackage[hidelinks,hypertexnames=false,hyperindex=true,pdfpagelabels,linktoc=all]{hyperref}
\usepackage{bookmark}
\makeatletter % to avoid hyperref warnings:
  \providecommand*{\toclevel@author}{999}
  \providecommand*{\toclevel@title}{0}
\makeatother

\usepackage{enumitem}

\theoremstyle{plain}
  \newtheorem{theorem}{Theorem}
  \newtheorem{proposition}{Proposition}
  
  \newtheorem{corollary}{Corollary}
\theoremstyle{definition}
  \newtheorem{definition}{Definition}
  
\theoremstyle{remark}
  \newtheorem{remark}{Remark}